\documentclass[11pt,a4paper]{article}
\usepackage[utf8]{inputenc}
\usepackage[T2A]{fontenc}
\usepackage[russian,english]{babel}
\usepackage{amsmath,amsfonts,amssymb,amsthm}
\usepackage{mathtools}
\usepackage{geometry}
\usepackage{tabularx}
\usepackage[bottom, perpage]{footmisc}
\usepackage{xcolor}
\usepackage{hyperref}
\usepackage{comment}
\usepackage{cleveref}
\newtheorem{theorem}{Theorem}[section]

\newtheorem{proposition}[theorem]{Proposition}

\theoremstyle{definition}
\newtheorem{definition}[theorem]{Definition}

\newtheorem{assumption}[theorem]{Assumption}

\newcommand{\F}{\mathcal{F}}

\newcommand{\Ghat}{\widehat{G}}

\newcommand{\norm}[1]{\|#1\|}

\newcommand{\ip}[2]{\langle #1, #2 \rangle}

\usepackage{titleps}
\newpagestyle{main}{%
\sethead[\thepage][][] 
        {On continuous embeddings of quantum Sobolev spaces}{}{\thepage} % Нечетные страницы (левый, центр, правый)
}
\title{\textbf{On сontinuous embeddings of Quantum Sobolev Spaces into Schatten classes}}
\author{Alexander Plakhotnikov \footnote{
\textsc{Student. Saint Petersburg State University, Department of Mathematics and Mechanics.} \\
\textit{Email adress:} \texttt{st132770@student.spbu.ru}}}
\date{\today}

\begin{document}

\maketitle

\begin{abstract}
This work investigates continuous embeddings for quantum Sobolev spaces $\mathfrak{H}_\gamma^{s,p}(G,H)$ into Schatten classes $S_r(H)$. We try to extend the results of Lakmon and Mensah to the case where the operators belong to Schatten classes $S_p(H)$ for $p \neq 2$. We establish that these quantum Sobolev spaces are Banach spaces and, by employing a duality argument, we define spaces for $p>2$. 
\bigskip

\end{abstract}

\section{Introduction}

The foundational work of Werner introduced a quantum Fourier transform linked to projective representations of locally compact abelian (LCA) groups, establishing a correspondence between operator ideals and function spaces \cite{Werner1984}. This framework has been systematically developed, for instance, by F\"uhr \cite{Fuehr2005} and more recently by Fulsche and Galke \cite{FulscheGalke2025}.

Within this operator-valued Fourier analysis, quantum Sobolev spaces emerge as a natural generalization of their classical counterparts. They are defined by imposing weighted integrability conditions on the quantum Fourier transform of an operator. For the case $p=2$, these spaces consist of Hilbert-Schmidt operators and are analogous to classical Bessel potential spaces \cite{LakmonMensah2025}. They provide the essential functional-analytic setting for studying operator-valued partial differential equations, such as quantum Poisson-type equations and generalized bosonic string equations \cite{GorkaKostrzewaReyes2015, Mensah2021}.

Recent work has extended the theory to various contexts, including compact groups and hypergroup Gelfand pairs \cite{BatakaEgweMensah2024, KumarKumar2023}. The work by Lakmon and Mensah \cite{LakmonMensah2025} specifically analyzed the structure and embedding properties of quantum Sobolev spaces of Hilbert-Schmidt operators.

The primary objective of this paper is to extend the results of Lakmon and Mensah \cite{LakmonMensah2025} to the case where the operators belong to Schatten classes $S_p(H)$ for $p \neq 2$. Our methodology adapts classical techniques, particularly the use of Hölder's inequality on the Fourier domain combined with the Hausdorff-Young inequality for the quantum Fourier transform.
\section{Preliminaries}

Let \(H\) be a separable complex Hilbert space. For a compact operator \(T\) on \(H\) denote by \((s_n(T))_{n\ge1}\) its singular values (nonincreasing). For \(0<p<\infty\) define the Schatten class \(S_p(H)\) by
\[
S_p(H) = \bigg\{ T\in K(H): \, \|T\|_{S_p}:= \bigg(\sum_{n}s_n(T)^p\bigg)^{1/p}<\infty\bigg\}
\]
Important special cases: \(S_1(H)\) (trace class), \(S_2(H)\) (Hilbert--Schmidt), \(S_\infty(H)=B(H)\) (bounded operators). For \(1\le p\le\infty\) the space \(S_p(H)\) is a Banach space; for \(p=2\) it is a Hilbert space with inner product \(\langle W,T\rangle_{S_2}=\mathrm{tr}(T^*W)\). Schatten classes satisfy Hölder-type inequalities: if \(1/p+1/q=1/r\) (with \(p,q,r\in[1,\infty]\)), then \(\|AB\|_{S_r}\le \|A\|_{S_p}\|B\|_{S_q}\) for appropriate \(A,B\).

Let $G$ be a locally compact abelian  group with Pontryagin dual $\Ghat$. Note that $\Ghat$ is also LCA. A projective unitary representation \(\pi:G\to U(H)\) with multiplier \(m\) yields, under additional integrability assumptions, a quantum Fourier transform \(F_U\) mapping operators to functions on \(\widehat G\).

\begin{definition}[~\cite{LakmonMensah2025}, Definition 2.1]
A map $\pi: G \to U(H)$ is called a \emph{projective representation with Heisenberg multiplier} if:
\begin{enumerate}
\item $\pi(x)\pi(y) = m(x,y)\pi(xy)$ for all $x,y \in G$
\item $m: G \times G \to \mathbb{T}$ is a Heisenberg multiplier
\item $m(x,y) = \overline{m(x^{-1}, y^{-1})}$ for all $x,y \in G$
\end{enumerate}
\end{definition}

\begin{definition}[\cite{LakmonMensah2025}, Definition 2.2]
For \( T \in S_1(H) \), the Fourier transform of \( T \) is given by
\begin{equation} \label{eq:fourier_transform}
        \F_U(T)(\xi) = \operatorname{tr}(T \pi(\xi)^*), \quad \xi \in \Ghat
\end{equation}    
Where \( \operatorname{tr}(A) \) is the trace of the operator \( A \). If \( T \in S_1(H) \) is such that \( \F_U(T) \in L^1(\Ghat) \), then the reconstruction formula is given by
    \begin{equation} \label{eq:reconstruction_formula}
        T = \int_{\Ghat} \F_U(T)(\xi) \,\pi (\xi) \, d\xi.
    \end{equation}
\end{definition}

\begin{assumption}[\cite{FulscheGalke2025}]
\label{ass:integrability}
We assume that the projective representation $\pi$ is \emph{integrable}: there exists a non-zero vector $\psi \in H$ such that
$$x \mapsto \ip{\pi(x)\psi}{\psi} \in L^1(G)$$
\end{assumption}

We have the following well-established results:

\begin{theorem}
\label{thm:qft-properties}
\begin{enumerate}
\item (~\cite{FulscheGalke2025}, Theorem 6.23) Assume that the representation is integrable. Then, the Fourier transform extends to a unitary operator $\mathcal{F}_U : S_2(H) \longrightarrow L^2(\Ghat)$ with adjoint $\mathcal{F}^{ -1}_U : L^2(\Ghat) \longrightarrow S_2(H)$

\item (~\cite{FulscheGalke2025}, Proposition 6.25) Assume that the representation is integrable. Let $p \in [1, 2]$ and let $q$ be such that $1/p+1/q =1$. Then

$$||\F_U(T)||_{L^q (\Ghat)} \leq ||T||_{S_p(H)} \qquad T \in S_p(H)$$

\item (Hausdorff-Young inequality, \cite{FulscheGalke2025}, Proposition 6.26) Assume that the representation is integrable. Let $p \in [1, 2]$ and let $q$ be such that $1/p+1/q =1$. Then 

$$ ||\mathcal{F}^{ -1}_U (f)||_{S_q(H)} \leq ||f||_{L^p(\Ghat)} \qquad f \in L^p(\Ghat) $$

\end{enumerate}
\end{theorem}

Fix a measurable weight function $\gamma: \Ghat \to (0,\infty)$ and smoothness parameter $s > 0$. 

\begin{definition}[see ~\cite{LakmonMensah2025}, Definition 3.1]
\label{def:sobolev-p-intermediate}
Let $1/p + 1/q = 1$. We call (quantum) Sobolev space the following set:
$$\mathfrak{H}_\gamma^{s,p}(G,H) := \bigg\{T \in S_p(H) : (1 + \gamma(\cdot)^2)^{s/2} \F_U(T) \in L^q(\Ghat)\bigg\}$$
with norm
$$\norm{T}_{\mathfrak{H}_\gamma^{s,p}} = \norm{(1 + \gamma(\cdot)^2)^{s/2} \F_U(T)}_{L^q(\Ghat)}$$
\end{definition}

We understand embedding in the sense of \cite{LakmonMensah2025}. 

\begin{assumption}
\label{ass:p-one}
We assume that $\F_U^{-1}: L^q(\Ghat) \to S_p(H)$ is bounded for relevant $q,p$.
\end{assumption}

\section{Main Results}

\begin{theorem}
\label{thm:topological-iso}
Let $1 < p < 2$. $\mathfrak{H}_\gamma^{s,p}(G, H)$ is a Banach space. The mapping \\$\Phi : T \mapsto (1+\gamma(\cdot)^2)^{s/2}\F_U(T)$ establishes a topological isomorphism between $\mathfrak{H}_\gamma^{s,p}(G, H)$ and its image in $L^q(\hat{G})$. The image of $\Phi$ is a closed subspace of $L^q(\hat{G})$, and $\Phi$ is a homeomorphism onto its image.
\end{theorem}

\begin{proof}

\medskip
1. Let $(T_n)$ be a Cauchy sequence in $\mathfrak{H}_\gamma^{s,p}(G,H)$. We have
\[
\|T_n - T_m\|_{\mathfrak{H}_\gamma^{s,p}}
  = \big\|(1+\gamma(\cdot)^2)^{s/2}\big(\F_U(T_n)-\F_U(T_m)\big)\big\|_{L^q(\widehat G)}
\]
Thus the condition that $(T_n)$ is Cauchy in $\mathfrak{H}_\gamma^{s,p}$ is 
\[
\forall\,\varepsilon>0\ \exists N\ \forall n,m\ge N:\ 
\big\|(1+\gamma(\cdot)^2)^{s/2}\big(\F_U(T_n)-\F_U(T_m)\big)\big\|_{L^q(\widehat G)}<\varepsilon
\]
In other words, the sequence
\[
\big((1+\gamma(\cdot)^2)^{s/2}\F_U(T_n)\big)_{n\in\mathbb N}
\]
is a Cauchy sequence in $L^q(\widehat G)$. Since $L^q(\widehat G)$ is complete, there exists
\[
g\in L^q(\widehat G)\quad\text{such that}\quad
(1+\gamma(\cdot)^2)^{s/2}\F_U(T_n)\to g \ \text{ in }L^q(\widehat G)
\]

\bigskip

\textsc{2: $\Phi$ is continuous.} This follows immediately from the definition: $\|\Phi(T)\|_{L^q} = \|T\|_{\mathfrak{H}_\gamma^{s,p}}$.

\bigskip

\textsc{$\Phi$ is injective.} If $\Phi(T) = 0$, then $(1 + \gamma(\cdot)^2)^{s/2} \F_U(T) = 0$ a.e. Since $\gamma(\xi) > 0$, we have $\F_U(T) = 0$ a.e. By injectivity of $\F_U$ on $S_p(H)$ (from Hausdorff-Young), we get $T = 0$.

\bigskip

\textsc{4.Image of $\Phi$ is closed in $L^q(\Ghat)$.}  
Let $(T_n)$ be a sequence in $\mathfrak{H}_\gamma^{s,p}(G,H)$ such that $\Phi(T_n)$ converges in $L^q(\hat{G})$ to some $g$. Since $(T_n)$ is bounded in the reflexive Banach space $S_p(H)$, by the Banach--Alaoglu theorem there exists a subsequence $(T_{n_k})$ weakly converging to some $T \in S_p(H)$. The continuity of $\Phi$ implies that $\Phi(T_{n_k})$ converges weakly to $\Phi(T)$ in $L^q(\hat{G})$. As the subsequence $\Phi(T_{n_k})$ also converges strongly to $g$, the weak limit must coincide with the strong limit, hence $\Phi(T) = g$. Therefore, $g$ lies in the image of $\Phi$, proving that the image is closed in $L^q(\hat{G})$.

Since $||T_n||_{\mathfrak{H}^{s,p}_\gamma} = ||f_n||_{L^q}$ is bounded, and $S_p(H)$ is reflexive  for $1 < p < \infty$, there exists a weakly convergent subsequence $(T_{n_k}) \rightarrow T \text{  in } S_p(H)$.

\bigskip

\textsc{5: $\Phi^{-1}$ is continuous on $\text{Im}(\Phi)$.} By the Open Mapping Theorem, since $\Phi: \mathfrak{H}_\gamma^{s,p}(G,H) \to \text{Im}(\Phi)$ is a continuous bijection between Banach spaces, it is a homeomorphism.
\end{proof}

Let $G$ be a locally compact abelian group with Haar measure $\mu$, and assume $G$ is $\sigma$-compact. Let $\pi$ be an integrable projective representation.  Fix $s>0$ and $1\le p'<2$, and set $q'$ by $1/p'+1/q'=1$. Let
$\gamma:\widehat G\to(0,\infty)$ be measurable and assume the local-boundedness condition $(1+\gamma^2)^{s/2}\in L^\infty_{\mathrm{loc}}(\widehat G).$ Denote by $C_0(\Ghat)$ continuous functions on $\Ghat$ with compact support.

Define the test family
\[
\mathcal{D}:=\big\{\,W_\varphi:=\mathcal{F}_U^{-1}\big((1+\gamma^2)^{s/2}\varphi\big):\ \varphi\in C_0(\widehat G)\,\big\}
\]
and equip $\mathcal D$ with the norm $\|W_\varphi\|_{\mathfrak{H}^{-s,p'}_\gamma}:=\|\varphi\|_{L^{q'}(\widehat G)}$. Then we have

\begin{proposition}
\label{prop:W2-construction}
\begin{enumerate}
  \item $\mathcal D$ is nonempty, and the map $\Psi:\mathcal D\to C_0(\widehat G)$,
    $\Psi(W_\varphi)=\varphi$, is a linear isometry from $(\mathcal D,\|\cdot\|_{\mathfrak{H}^{-s,p'}_\gamma})$ onto $C_0(\widehat G)$ equipped with the $L^{q'}$--norm.
  \item The completion
    \[
      \mathfrak{H}^{-s,p'}_\gamma:=\overline{\mathcal D}_{\|\cdot\|_{\mathfrak{H}^{-s,p'}_\gamma}}
    \]
    is a (nontrivial) Banach space. Under the standing hypothesis $C_0(\widehat G)$ is dense in $L^{q'}(\widehat G)$, hence $\mathfrak{H}^{-s,p'}_\gamma$ is isometrically isomorphic to $L^{q'}(\widehat G)$.
  
\end{enumerate}
\end{proposition}

\begin{proof}
\emph{(Well-definedness of $\mathcal D$.)}
Let $\varphi\in C_0(\widehat G)$, nonzero, with compact support $K=\operatorname{supp}\varphi$. By the local-boundedness hypothesis $(1+\gamma^2)^{s/2}$ is essentially bounded on $K$, hence
\[
g:=(1+\gamma^2)^{s/2}\varphi\in L^2(\widehat G)\cap L^{q'}(\widehat G).
\]
Since $\mathcal F_U^{-1}$ is defined on $L^2(\widehat G)$, the inverse transform
$W_\varphi:=\mathcal F_U^{-1}g$ exists and lies in $S_2(H)$. By construction
\[
(1+\gamma^2)^{-s/2}\mathcal F_U(W_\varphi)=\varphi\in L^{q'}(\widehat G),
\]
so $W_\varphi\in\mathcal D$. Thus $\mathcal D\neq\varnothing$.

\emph{(Isometry onto $C_0$.)}
The map $\Psi:\mathcal D\to C_0(\widehat G)$, $\Psi(W_\varphi)=\varphi$, is linear and injective: if $\Psi(W_\varphi)=0$ a.e.\ then $g=0$ a.e.\ and $W_\varphi=0$. By definition
$\|W_\varphi\|_{\mathfrak{H}^{-s,p'}_\gamma}=\|\varphi\|_{L^{q'}}$, hence $\Psi$ is an isometry onto its image $C_0(\widehat G)$ with the $L^{q'}$--norm.

\emph{(Completion.)}
Because $\Psi$ is an isometry, the completion $\mathfrak{H}^{-s,p'}_\gamma$ of $\mathcal D$ is isometrically isomorphic to the closure of $C_0(\widehat G)$ in $L^{q'}(\widehat G)$. Under the $\sigma$-compactness hypothesis $C_0(\widehat G)$ is dense in $L^{q'}(\widehat G)$ for $1\le q'<\infty$, so the closure equals $L^{q'}(\widehat G)$. Therefore $\mathfrak{H}^{-s,p'}_\gamma$ is (isometrically) $L^{q'}(\widehat G)$ and in particular a Banach space; it is nontrivial since $C_0(\Ghat)$ contains nonzero functions.

\end{proof}

Let $s>0$ and $p>2$, and set $p'=\dfrac{p}{p-1}<2$. Let the conditions be fulfilled under which
$\mathfrak{H}^{-s,p'}_{\gamma}(G,H)$ is a Banach space. Equip $\mathfrak{H}^{-s,p'}_{\gamma}(G,H)$ with the standart norm  \ref{def:sobolev-p-intermediate} and define the bilinear pairing
\[
\langle T, W\rangle := \operatorname{tr}(T W^*),\qquad T\in S_p(H),\; W\in \mathfrak{H}^{-s,p'}_{\gamma}(G,H)
\]
Then we have

\begin{proposition}
\label{prop:pairing-duality}

\begin{enumerate} 

\item[(i)] For every fixed $T\in S_p(H)$ the linear functional
      $L_T : \mathfrak{H}^{-s,p'}_{\gamma}\to\mathbb{C}$, $L_T(W):=\langle T,W\rangle$, is continuous.
      More precisely, there exists a constant $C>0$ (depending only on $\gamma,s,p$ and the
      Hausdorff--Young constants) such that
      \[
      |\langle T,W\rangle| \le C\,\|T\|_{S_p}\,\|W\|_{\mathfrak{H}^{-s,p'}_{\gamma}},
      \qquad\forall W\in \mathfrak{H}^{-s,p'}_{\gamma}
      \]
      Hence the map $\iota: S_p(H)\to\bigl(\mathfrak{H}^{-s,p'}_{\gamma}\bigr)^*$, $\iota(T)=L_T$, is
      well defined and continuous.

\item[(ii)] The pairing is non-degenerate: if $T\in S_p(H)$ satisfies
      $\langle T,W\rangle=0$ for all $W\in \mathfrak{H}^{-s,p'}_{\gamma}$, then $T=0$; similarly,
      if $W\in \mathfrak{H}^{-s,p'}_{\gamma}$ satisfies $\langle T,W\rangle=0$ for all $T\in S_p(H)$,
      then $S=0$.

\item[(iii)] Let $\mathfrak{H}^{s,p}_\gamma$ be the closure of the image of $S_p(H)$ under $\iota$ in $(\mathfrak{H}^{-s,p'}_\gamma)^*$ endowed with the operator norm inherited from $\bigl(\mathfrak{H}^{-s,p'}_{\gamma}\bigr)^*$, yields a Banach space which is canonically and continuously embedded into $S_p(H)$ via the inverse of $\iota$ on its image. In particular every element of $\mathfrak{H}^{s,p}_{\gamma}$ can be represented by a (unique) Schatten operator and the stated norm inequality in (i) gives the continuous embedding.
\end{enumerate}
\end{proposition}

\begin{proof}

\medskip\noindent\textbf{1.}
Fix $W\in \mathfrak{H}^{-s,p'}_{\gamma}$. By the definition of $\mathfrak{H}^{-s,p'}_{\gamma}$ we have
\[
(1+\gamma(\cdot)^2)^{-s/2}\F_U(W) \in L^{q'}(\widehat G)
\]
Using the Hausdorff--Young bound   from Theorem \ref{thm:qft-properties} (item 3) (recall $1\le p'\!<2$,
$1/p'+1/q'=1$), we obtain a control of the Schatten $(p')$–norm of $W$ in terms of the
weighted $L^{q'}$–norm: there exists a constant $C_1=C_1(\gamma,s,p)$ such that
\[
\|W\|_{S_{p'}} \le C_1\,\big\|(1+\gamma^2)^{-s/2}\F_U(W)\big\|_{L^{q'}} = C_1\,\|W\|_{\mathfrak{H}^{-s,p'}_{\gamma}}
\]

\medskip\noindent\textbf{2.}
Let $T\in S_p(H)$ be fixed. By Hölder's inequality for Schatten classes 
we have
\[
|\operatorname{tr}(T W^*)| \le \|T\|_{S_p}\,\|W\|_{S_{p'}}.
\]
Combining this with the estimate from Step 1 yields
\[
|\langle T,W\rangle| = |\operatorname{tr}(T W^*)|
\le \|T\|_{S_p}\,\|W\|_{S_{p'}}
\le C_1\,\|T\|_{S_p}\,\|W\|_{\mathfrak{H}^{-s,p'}_{\gamma}}
\]
Thus $L_T$ is continuous on $\mathfrak{H}^{-s,p'}_{\gamma}$ and
$\|L_T\|_{(\mathfrak{H}^{-s,p'}_{\gamma})^*}\le C_1\|T\|_{S_p}$. This proves claim (i) with $C=C_1$.

\medskip\noindent\textbf{3.}
Assume $T\in S_p(H)$ satisfies
\[
  \operatorname{tr}\bigl(T\,W^*\bigr) \;=\; 0
  \quad\forall W\in \mathfrak{H}^{-s,p'}_\gamma(G,H).
\]
We wish to show $T=0$. Recall from Proposition 3.2 that the test family
\[
  \mathcal{D}
  \;=\;\bigl\{\,W_\varphi = \mathcal \F_U^{-1}\bigl((1+\gamma^2)^{-s/2}\,\varphi\bigr)
    : \varphi\in C_0(\hat G)\bigr\}
  \;\subset\; \mathfrak{H}^{-s,p'}_\gamma(G,H)
\]
is \emph{dense} in $\mathfrak{H}^{-s,p'}_\gamma(G,H)$ and contains all finite‐rank operators $\mathcal{R}$.
Indeed, if $\varphi\in C_0(\hat G)$, then $W_\varphi$ is obtained by inverse Fourier
transform of a compactly supported function, hence of finite rank. Therefore: $\mathcal{R}\subset\;\mathcal{D}$ and $\overline{\mathcal{R}}_{\|\cdot\|_{\mathfrak{H}^{-s,p'}_\gamma}}=\mathfrak{H}^{-s,p'}_\gamma(G,H).$

By hypothesis, $\operatorname{tr}\bigl(T\,W^*\bigr) \;=\; 0 $ $\forall\,W\in \mathfrak{H}^{-s,p'}_\gamma(G,H).$
Restricting to $W\in\mathcal{R}$ implies $\operatorname{tr}\bigl(T\,R^*\bigr) \;=\;0
  \quad\forall\,R\in \mathcal{R}.$
Since $\mathcal{R}$ is dense in the Schatten class $S_{p'}(H)$ and the trace pairing $(T,R)\mapsto\operatorname{tr}(T\,R^*)$ is continuous on $S_p\times S_{p'}$,  it follows by continuity that
\[
  \operatorname{tr}\bigl(T\,R^*\bigr) =0
  \quad\forall\,R\in S_{p'}(H).
\]
This forces $T=0$. The same density argument applies if one assumes $\operatorname{tr}(T\,W^*)=0$ for all $T\in S_p(H)$ and wishes to conclude $W=0$.

\medskip\noindent\textbf{4.}
By Step 2 we have a continuous linear map
\(\iota: S_p(H)\to (\mathfrak{H}^{-s,p'}_{\gamma})^*,\; \iota(T)=L_T.\)
Define \(\mathfrak{H}^{s,p}_{\gamma}\) to be the closure of \(\iota(S_p(H))\) in \((\mathfrak{H}^{-s,p'}_{\gamma})^*\)
with the operator norm. Since \((\mathfrak{H}^{-s,p'}_{\gamma})^*\) is Banach, its closed subspace
\(\mathfrak{H}^{s,p}_{\gamma}\) is Banach as well. The estimate in Step 2 gives the continuous
injection \(S_p(H)\hookrightarrow \mathfrak{H}^{s,p}_{\gamma}\) (via $\iota$), and every element of
$\mathfrak{H}^{s,p}_{\gamma}$ is represented by a (unique) continuous functional on $\mathfrak{H}^{-s,p'}_{\gamma}$
which, by Riesz-style identification through the trace pairing, corresponds to a (unique)
Schatten operator on $H$. Thus $\mathfrak{H}^{s,p}_{\gamma}$ canonically embeds into $S_p(H)$ and the
claimed norm inequality holds.
\end{proof}

In analogy with the classical situation, we distinguish between the
\emph{inhomogeneous} quantum Sobolev space
\[
\mathfrak{H}^{s,p}_\gamma(G,H)
= \Bigl\{ T\in S_p(H): (1+\gamma(\cdot)^2)^{s/2} \,\mathcal{F}_U(T)\in L^q(\widehat G)\Bigr\}
\]
with norm
\[
\|T\|_{\mathfrak{H}^{s,p}_\gamma}
= \|(1+\gamma^2)^{s/2}\,\mathcal{F}_U(T)\|_{L^q}
\]
and the \emph{homogeneous} quantum Sobolev space
\[
\dot{\mathfrak{H}}^{s,p}_\gamma(G,H)
= \Bigl\{ T\in S_p(H): \gamma(\cdot)^s \,\mathcal{F}_U(T)\in L^q(\widehat G)\Bigr\}
\]
with norm
\[
\|T\|_{\dot{\mathfrak{H}}^{s,p}_\gamma}
= \|\gamma^s\,\mathcal{F}_U(T)\|_{L^q}.
\]

Let $1<p\le2$, $s>0$, and set $q=\dfrac{p}{p-1}>2$.  Fix $\alpha>0$ and assume one of the following
two verifiable integrability hypotheses holds:
\begin{enumerate}
\item\emph{(Inhomogeneous case)} \quad $(1+\gamma(\cdot)^2)^{-s/2}\in L^\alpha(\widehat G)$.
\item\emph{(Homogeneous case)} \quad $\gamma(\cdot)^{-s}\in L^\alpha(\widehat G)$.
\end{enumerate}

Define
\[
\sigma:=\frac{\alpha q}{\alpha+q}\qquad\text{and}\qquad
\beta:=\frac{\alpha q}{\alpha(q-1)-s}.
\]
Assume moreover that $1<\sigma\le 2$. Then we have

\begin{proposition}
\label{thm:main-embedding-corrected}
\begin{itemize}
\item[(i)] If (1) is satisfied and $T\in \mathfrak{H}^{s,p}_\gamma(G,H)$, then $\F_U(T)\in L^\sigma(\widehat G)$ and $T\in S_{\beta}(H), \qquad \|T\|_{S_\beta}\le C\,\|T\|_{\mathfrak{H}^{s,p}_\gamma}$ for a constant $C$ depending only on $\alpha,q,s,\gamma$ and the Hausdorff--Young constant $C_{HY}$.

\item[(ii)] If (2) is satisfied and $T\in\dot{\mathfrak{H}}^{s,p}_{\gamma(G,H)}$, then $\F_U(T)\in L^\sigma(\widehat G)$ and $T\in S_{\beta}(H),$ \\$  \|T\|_{S_\beta}\le C\,\|T\|_{\dot{\mathfrak{H}}^{s,p}_\gamma}$
with the same type of constant $C$.
\end{itemize}
\end{proposition}

\begin{proof}
We treat both cases by the same argument. Take $T$ in the relevant Sobolev space (either $T\in \mathfrak{H}^{s,p}_\gamma$ or
$T\in\dot{\mathfrak{H}}^{s,p}_\gamma$). Set the multiplier \(m(\xi)\) to be
\[
m(\xi)=
\begin{cases}
(1+\gamma(\xi)^2)^{-s/2} &\text{inhom. case }\\[4pt]
\gamma(\xi)^{-s} &\text{hom. case }
\end{cases}
\]

By the respective hypothesis we have \(m\in L^\alpha(\widehat G)\). Now write pointwise
\[
\F_U(T)(\xi)= m(\xi)\cdot \bigg( m(\xi)^{-1}\F_U(T)(\xi)\bigg)
\]
where \(m^{-1}\F_U(T)\) is exactly the Fourier–side quantity whose \(L^q\)-norm
coincides with the Sobolev norm of \(T\) (either \((1+\gamma^2)^{s/2}\F_U(T)\)
or \(\gamma^s \F_U(T)\)). Apply Hölder's inequality on \(\widehat G\) with
exponents \(\alpha\) and \(q\) chosen so that
\[
\frac{1}{\sigma}=\frac{1}{\alpha}+\frac{1}{q}
\]
By the definition \(\sigma=\dfrac{\alpha q}{\alpha+q}\) this identity holds.
Hence
\[
\|\F_U(T)\|_{L^\sigma}
\le \|m\|_{L^\alpha}\,\|m^{-1}\F_U(T)\|_{L^q}
= C_1\,\|T\|_{\mathcal H} \quad \mathcal{H} =  {\mathfrak{H}^{s,p}_\gamma} \text{ or } {\dot{\mathfrak{H}}^{s,p}_\gamma}
\]

Since by assumption \(1<\sigma\le2\), the operator form of the
Hausdorff--Young inequality gives a bounded map
\[
\mathcal{F}_U^{-1}:L^\sigma(\widehat G)\longrightarrow S_{\sigma'}(H),
\qquad \sigma'=\frac{\sigma}{\sigma-1}
\]
Applying that to \(f=\F_U(T)\) yields
\[
\|T\|_{S_{\sigma'}} \le C_{\mathrm{HY}}\,\|\F_U(T)\|_{L^\sigma}
\le C_{\mathrm{HY}}\,C_1\,\|T\|_{\mathcal H}
\]
Finally compute the exponent \(\sigma'\). Substituting
\(\sigma=\dfrac{\alpha q}{\alpha+s}\) gives
\[
\sigma'=\frac{\sigma}{\sigma-1}=\frac{\alpha q}{\alpha(q-1)-s}=\beta.
\]
Combining the displayed estimates produces $\|T\|_{S_\beta}\le C\,\|T\|_{\mathcal H}$ with \(C=C_{\mathrm{HY}}C_1\), which is the asserted norm estimate and hence the claimed continuous embedding.
\end{proof}

Let $G$ be a locally compact abelian group and let $\pi$ be a projective representation
satisfying the integrability (or square--integrability) assumptions and the Heisenberg
multiplier conditions used in ~\cite{LakmonMensah2025}. Fix $s>0$ and $p>2$, and denote $p' = p/(p-1)<2$.
Assume the weight $\gamma:\widehat G\to(0,\infty)$ satisfies the integrability condition so that $\mathfrak{H}^{-s,p'}_\gamma(G,H)$ is well defined (by Theorem~\ref{thm:negative-well-defined}) as a Banach space. Then we have

\begin{proposition}
\label{thm:improved-4.1}
In general one cannot upgrade the embedding above to
\[
   \mathfrak{H}^{s,p}_\gamma(G,H)\hookrightarrow S_\beta(H), \qquad \beta<p,
\]
without further assumptions on $\gamma$ or on the representation $\pi$. In fact, for natural models such as $G=\mathbb R^n$ with the Schr\"odinger/Heisenberg representation, there exists a bounded sequence in $\mathfrak{H}^{s,p}_\gamma(G,H)$ whose $S_\beta$-norms diverge for every $\beta<p$.
\end{proposition}

\begin{proof}
Consider the model case $G=\mathbb R^n$ with the standard Heisenberg representation. Fix $q$ associated to $p'$, and construct functions $a_n$ supported on sets of measure $\varepsilon_n\to0$, with
\[
   \|a_n\|_{L^q}=1 \quad \text{ but} \quad \|a_n\|_{L^\rho}\to\infty \quad
   \text{for all }\rho>q
\]
Such functions can be constructed explicitly, e.g. $a_n = \varepsilon_n^{-1/q} \textbf{1}_{E_n}$ with $|E_n|=\varepsilon_n$, so that $||a_n||_{L^q} =1$ and $||a_n||_{L^\rho} = \varepsilon_n^{1/\rho -1/q} \mapsto \infty$ for $\rho > q$. Let $T_n:=\F_U^{-1}(a_n)$. Then $\|T_n\|_{\mathfrak{H}^{s,p}_\gamma}$ is uniformly bounded ($L^q$-norm of $a_n$ is fixed), while for any $\beta<p$ Hausdorff--Young requires $\rho>q$ to control $\|T_n\|_{S_\beta}$. Hence $\|T_n\|_{S_\beta}\to\infty$ as $n\to\infty$, showing that no continuous embedding into $S_\beta$ is possible.
\end{proof}

The obstruction does not preclude improved embeddings under stronger
assumptions. For instance, if $(1+\gamma(\xi)^2)^{-r/2}\in L^1(\widehat G)$ for $r$ large enough, or if $\pi$ enjoys enhanced $L^\rho\to S_{\rho'}$ bounds with $\rho>2$, then one can obtain continuous embeddings $\mathfrak{H}^{s,p}_\gamma\hookrightarrow S_\beta$ for some $\beta<p$. Precise thresholds can be derived using the same weighted H\"older/Hausdorff--Young estimates as before.

\end{document}